\documentclass[12pt]{amsart}
\usepackage{enumerate}
\usepackage{amsmath, amssymb, amsthm, amsfonts}
\usepackage{tikz}
\usepackage{geometry}
\usepackage[pdftex, bookmarks=true]{hyperref}
\usepackage{bbm} %for the 1 symbol of indicator functions
\usepackage{url}
\usepackage{cmap} %supports pdf copy-paste
\usepackage{verbatim}  %for comment environment
\usepackage{bm}  %for bold in math mode

\usepackage{marginnote}

\newtheorem{theorem}{Theorem}[section]
\newtheorem{proposition}[theorem]{Proposition}

\newtheorem{lemma}[theorem]{Lemma}
\newtheorem{conjecture}[theorem]{Conjecture}

\newtheorem*{claim*}{Claim}

\theoremstyle{definition}
\newtheorem{remark}[theorem]{Remark}
\newtheorem*{remark*}{Remark}

\newtheorem{definition}[theorem]{Definition}

\newtheorem{observation}[theorem]{Observation}

\newcommand\ds{\displaystyle}

\newcommand\Eb{\mathbb{E}}  %expectation
\newcommand\Pb{\mathbb{P}}  %probability

\newcommand\Zb{\mathbb{Z}}
\newcommand\Nb{\mathbb{N}}
\newcommand\Gb{\mathbb{G}}

\newcommand\Gc{\mathcal{G}}

\newcommand\Oc{\mathcal{O}}  %big O
\newcommand\dhat{\hat{d}}
\newcommand\comp{\mathrm{c}}
\newcommand\phihat{\hat{\varphi}}

\newcommand\al{\alpha}
\newcommand\be{\beta}

\newcommand\eps{\varepsilon}

\newcommand\alstar{\alpha^{\star}}

\newcommand\kappastar{\kappa^{\star}}
\newcommand\alfm{\alpha^{\mathrm{FM}}}
\newcommand\alfc{\alpha^{\mathrm{FrC}}}
\newcommand\alrs{\alpha^{\mathrm{RS}}}
\newcommand\alrsb{\alpha^{\mathrm{1-RSB}}}

\newcommand\bemax{\beta_{\mathrm{max}}}
\newcommand\kind{k^\mathrm{ind}}
\newcommand\muedge{\mu_\mathrm{edge}}
\newcommand\muvertex{\mu_\mathrm{vertex}}

\newcommand\sm{\setminus}

\DeclareMathOperator{\dist}{dist}

\newcommand{\defeq}{\mathrel{\vcenter{\baselineskip0.5ex \lineskiplimit0pt
                     \hbox{\scriptsize.}\hbox{\scriptsize.}}}%
                     =}

\DeclareFontFamily{U}{matha}{\hyphenchar\font45}
\DeclareFontShape{U}{matha}{m}{n}{
  <-6> matha5 <6-7> matha6 <7-8> matha7
  <8-9> matha8 <9-10> matha9
  <10-12> matha10 <12-> matha12
  }{}

\DeclareSymbolFont{matha}{U}{matha}{m}{n}
\DeclareMathSymbol{\Lt}{3}{matha}{"CE}

\title[Star decompositions and independent sets]{Star decompositions and independent sets\\in random regular graphs}

\author{Viktor Harangi}
\address{HUN-REN Alfr\'ed R\'enyi Institute of Mathematics, Budapest, Hungary}
\email{harangi@renyi.hu}
\thanks{The author was supported by the MTA-R\'enyi Counting in Sparse Graphs ``Momentum'' Research Group, by NRDI 
(grant KKP 138270), and by the Hungarian Academy of Sciences (J\'anos Bolyai Scholarship).}

\begin{document}

\begin{abstract}
A $k$-star decomposition of a graph is a partition of its edges into $k$-stars (i.e., $k$ edges with a common vertex). The paper studies the following problem: for what values of $k>d/2$ does the random $d$-regular graph have a $k$-star decomposition (asymptotically almost surely, provided that the number of edges is divisible by $k$)?

Delcourt, Greenhill, Isaev, Lidick\'y, and Postle proposed the following conjecture. It is easy to see that a $k$-star decomposition necessitates the existence of an independent set of density $1-d/(2k)$. So let $k^{\mathrm{ind}}_d$ be the largest $k$ for which the random $d$-regular graph a.a.s.~contains an independent set of this density. Clearly, $k$-star decompositions cannot exist for $k>k^{\mathrm{ind}}_d$. The conjecture suggests that this is essentially the only restriction: there is a threshold $k^\star_d$ such that $k$-star decompositions exist if and only if $k \leq k^\star_d$, and it (basically) coincides with the other threshold, i.e., $k^\star_d \approx k^{\mathrm{ind}}_d$.

We confirm this conjecture for sufficiently large $d$ by showing that a $k$-star decomposition exists if $d/2< k < k^{\mathrm{ind}}_d$. In fact, we prove the existence even if $k=k^{\mathrm{ind}}_d$ for degrees $d$ with asymptotic density $1$. 
\end{abstract}
% MSC codes
% 05C80  	Random graphs (graph-theoretic aspects)
% 05C69  	Vertex subsets with special properties (dominating sets, independent sets, cliques, etc.)

%%%%%%%%%%%

\maketitle

%%%%%%%%%%%%%%%%%%%%%%%%%%%%%%%%%%%%%%%%%%%%%
\section{Introduction} \label{sec:intro}

For a positive integer $d \geq 3$, let $\Gc_{N,d}$ denote the $N$-vertex random $d$-regular graph, that is, a uniform random graph among all simple $d$-regular graphs on the vertex set $\{1,\ldots, N\}$. We say that $\Gc_{N,d}$ \emph{asymptotically almost surely} (a.a.s.~in short) has a property if the probability that $\Gc_{N,d}$ has this property converges to $1$ as $N \to \infty$. 

Given an integer $k \geq 2$, it is natural to ask whether the edges of $\Gc_{N,d}$ can be partitioned into edge-disjoint stars, each containing $k$ edges. Here we need to restrict ourselves to those $N$ for which the number of edges ($Nd/2$) is divisible by $k$. If such a partition exists with probability $1-o_N(1)$, then we say that $\Gc_{N,d}$ a.a.s.~has a $k$-star decomposition.

Star decompositions are expected to exist a.a.s.~for every $k$ in the range $k \leq d/2$. This was rigorously proven for odd $k$ in \cite{delcourt2023decomposing} using previous results about so-called $\be$-orientations.\footnote{It was incorrectly claimed in \cite{delcourt2018random} that the case $k \leq d/2$ is fully solved, which was later clarified in \cite{delcourt2023decomposing}.}

As for the range $k>d/2$, it was pointed out in \cite{delcourt2023decomposing} that the vertices that are not the center of any star 
%in the decomposition 
form an independent set of density\footnote{By density we mean the size of the set divided by the total number of vertices in the graph.}
\begin{equation} \label{eq:al_dk}
\al_{d,k} \defeq 1 - \frac{d}{2k} . 
\end{equation}
Indeed, given a $k$-star decomposition, in each star we may direct the edges away from the center. The resulting orientation of the graph is such that each out-degree is $0$ or $k$, because each vertex can be the center of at most one star due to $k>d/2$. Hence the number of vertices with out-degree $k$ is equal to the number of stars, which is clearly $Nd/(2k)$. It follows that the complement (i.e., the set of vertices with out-degree $0$) has density $\al_{d,k}$. Finally, it holds trivially for any orientation that vertices with out-degree $0$ form an independent set.

This simple observation links the problem of star decompositions to the widely-studied topic of independent sets in random regular graphs. In particular, a $k$-star decomposition may only exist if $\Gc_{N,d}$ a.a.s.~contains an independent set of size $\al_{d,k} N$. It is known \cite{bayati2013interpolation} that for each $d \geq 3$ there exists an $\alstar_d$ such that, for any $\eps>0$, the independence ratio of $\Gc_{N,d}$ is a.a.s.~$\eps$-close to $\alstar_d$. In other words, the independence ratio of $\Gc_{N,d}$ converges to $\alstar_d$ in probability.

For convenience, we introduce the inverse function of $k \mapsto \al_{d,k}=1-d/(2k)$:
\begin{equation} \label{eq:kappa}
\kappa_d(\al) \defeq \frac{d}{2(1-\al)} 
\quad \text{so that} \quad 
\kappa_d( \al_{d,k} ) = k .
\end{equation}
Furthermore, let   
\begin{equation} \label{eq:k_ind}
\kappastar_d \defeq \kappa_d(\alstar_d) 
\quad \text{and} \quad 
\kind_d = \lfloor \kappastar_d \rfloor .
\end{equation}
The point is that for $k \geq \kind_d+1 > \kappastar_d$ we have $\al_{d,k} > \alstar$. Therefore $\Gc_{N,d}$ a.a.s.~does not contain an independent set of density $\al_{d,k}$, and hence it cannot have a $k$-star decomposition, either. Is this the only restriction? Is $\kind_d$ the threshold\footnote{Normally we can be sure that $\Gc_{N,d}$ a.a.s.~contains an independent set of density $\al_{d,k}$ for $k=\kind_d = \lfloor \kappastar_d \rfloor$. In principle, it may happen that $\kappastar_d$ is an integer, in which case this is only guaranteed for $k=\kappastar_d-1$. It would probably be safe to assume that this is not the case for any $d$.} for the star decomposition problem as well? The following is a slight variant of \cite[Conjecture 1.1]{delcourt2023decomposing}.
\begin{conjecture} \label{conj:sd}
Let $d \geq 3$. If $d/2 < k \leq \kind_d$, then $\Gc_{N,d}$ a.a.s.~has a $k$-star decomposition as $N \to \infty$ with $Nd$ being divisible by $2k$ (apart from $d=5$ and perhaps a small number of further exceptional degrees).
\end{conjecture}
Asymptotically we have $\alstar_d \sim 2\log(d)/d$, and hence 
\[ \kind_d \sim \frac{d}{2} + \log d 
\quad \text{as } d \to \infty .\]
In \cite{delcourt2023decomposing} ``one-sixth'' of the conjecture was proven: they showed the a.a.s.~existence of $k$-star decompositions in the range 
\[ \frac{d}{2} \, < k \leq \, \frac{d}{2}+\frac{1}{6}\log d .\]
The method in \cite{delcourt2023decomposing} works well for small values of $d$: for $d\leq 100$ it actually covers the whole range $(d/2, \kind_d]$ except maybe the endpoint $\kind_d$.

In this paper we improve $1/6$ to the optimal constant $1$. In fact, we prove that $k$-star decompositions exist up to $k= \kind_d - 1$ if $d$ is sufficiently large, and even for $k= \kind_d$ for most degrees. This fully settles \cite[Conjecture 1.1]{delcourt2023decomposing} for large $d$.
% XXX what if k does not divide Nd/2
%
\begin{theorem} \label{thm:main}
For sufficiently large $d$ and for $d/2 < k \leq \kind_d-1$, the random $d$-regular graph $\Gc_{N,d}$ asymptotically almost surely has a $k$-star decomposition (as $N \to \infty$ with $Nd$ being divisible by $2k$). Even when the divisibility condition is not satisfied, we still have the following: $\Gc_{N,d}$ a.a.s.~contains edge-disjoint $k$-stars covering all but at most $k-1$ edges of the graph.

Moreover, this holds even for $k=\kind_d$ provided that 
\begin{equation} \label{eq:frac_part}
\{ \kappastar_d \}= \left\{ \frac{d}{2(1-\alstar_d)} \right\} 
> \frac{(\log d)^3}{d} ,
\end{equation}
where $\{\cdot\}$ denotes the fractional part and $\alstar_d$ denotes the asymptotic independence ratio of $\Gc_{N,d}$. Furthermore, condition \eqref{eq:frac_part} is satisfied by degrees $d$ with asymptotic density $1$.
\end{theorem}
Besides Theorem~\ref{thm:main} (concerned with the case of large degrees), we also tested our approach for specific values of $d$. There is a conjecture regarding the precise value of $\alstar_d$ for $d \geq 20$, and we worked under the assumption that this conjecture holds true in order to see if our approach could potentially prove the existence of $k$-star decompositions up to $k=\kind_d$ for a given $d$ in the range $30 \leq d \leq 3000$. Surprisingly, the answer is positive apart from a small number of exceptional degrees; see Section~\ref{sec:specific} for details.

\subsection*{Proof outline}
In \cite{delcourt2023decomposing} the so-called small subgraph conditioning method was used. (Actually, \cite{delcourt2023decomposing} is an extension of an earlier paper \cite{delcourt2018random}, where $\Gc_{N,4}$ was decomposed into $3$-stars.) At the heart of that approach lies a second moment calculation, the completely rigorous treatment of which leads to very technical computations. 

In comparison, our approach only relies on first moment calculations. This is due to the fact that our method produces star decompositions from independent sets using \emph{in-degree-regular orientations} along the way. Next we give a brief outline of the proof.

For a fixed $d$ and $k>d/2$, the following are equivalent objects to consider for a $d$-regular simple graph $G$:
\begin{itemize}
\item a $k$-star decomposition of $G$;
\item an orientation of $G$ whereby each out-degree is $0$ or $k$;
\item an independent set $A$ of density $\al_{d,k}=1-d/(2k)$ along with a \emph{$(d-k)$-in-regular orientation} of the induced subgraph $G[A^\comp]$, that is, an orientation of $G[A^\comp]$ for which each in-degree is exactly $d-k$, where $A^\comp \defeq V(G) \sm A$.
\end{itemize}
Note that a graph $H$ has an $\ell$-in-regular orientation if and only if its average degree is $2 \ell$ and every induced subgraph of $H$ has average degree at most $2 \ell$ (see Proposition~\ref{prop:reg_orientation}). So we need to find an independent set $A$ in $\Gc=\Gc_{N,d}$ for which $H=\Gc[A^\comp]$ satisfies this condition.

First we try to find an $A$ with the property that any vertex outside $A$ has a limited number of neighbors in $A$ (at most $\dhat \approx \tau d$ neighbors for some constant $0<\tau<1/2$). We will call such independent sets \emph{$\dhat$-thin}, see Definition~\ref{def:thin}. We start from an independent set $A$ with density $\alstar_d-\eps$. Then we keep removing vertices from $A$ that are neighbors of outside vertices with too many edges going to $A$. Here we need delicate first moment bounds in $\Gc_{N,d}$ to show that the total number of removed vertices is small (with high probability). Specifically, in Section~\ref{sec:thin} we prove for large $d$ that $\Gc_{N,d}$ a.a.s.~has a $\tau d$-thin independent set of density at least $\alstar_d - (\log d)^3/d^2$; see Lemma~\ref{lem:thin}(b). Then we can choose $k$ in such a way that $\al_{d,k}$ is below this density.

Once we have a thin independent set of appropriate density, we will use Lemma~\ref{lem:sd}. It is a (deterministic) result about $d$-regular graphs that produces star decompositions from thin independent sets. In order to check the conditions of the lemma in our setting, we need an upper bound on the average degree of induced subgraphs of $\Gc_{N,d}$. For the purposes of Theorem~\ref{thm:main} a very basic bound (see Proposition~\ref{prop:weak}) will suffice.

\subsection*{Notations}

As usual, $V(G)$ denotes the vertex set of a graph $G$, and we write $\deg_G(v)$ or simply $\deg(v)$ for the degree of a vertex $v$, while $G[U]$ stands for the induced subgraph on $U \subseteq V(G)$. By \emph{density} we always refer to the relative size $|U| \big/ |V(G)|$ of a subset $U$. Furthermore, when $G$ is clear from the context, we use the following shorthand notations.
\begin{itemize}
\item Complement: $U^\comp \defeq V(G) \sm U$.
\item Edge count: $e(G)$ denotes the total number of edges, while $e[U] \defeq e\big( G[U] \big)$ is the number of edges inside $U$. Also, for disjoint subsets $U,U' \subseteq V(G)$ we write $e[U,U']$ for the number of edges between $U$ and $U'$. Finally, let $e[v,U] \defeq e[\{v\},U]$.
\end{itemize}
Throughout the paper, the function $h$ stands for $h(x)=-x \log x$.

\subsection*{Organization of the paper}
Section~\ref{sec:star_decomp} contains the basic idea behind our approach: we introduce thin independent sets and show how they can be used to find star decompositions.  
Section~\ref{sec:ind_ratio} briefly summarizes what is known about the independence ratio of the random $d$-regular graph $\Gc_{N,d}$, while Section~\ref{sec:first_moment} discusses the first moment method in $\Gc_{N,d}$. In Section~\ref{sec:thin} we prove the existence of thin independent sets in $\Gc_{N,d}$. The proof of Theorem~\ref{thm:main} is given in Section~\ref{sec:main_proof}, while the case of specific degrees is considered in Section~\ref{sec:specific}. 

%%%%%%%%%%%%%%%%%%%%%%%%%%%%%%%%%%%%%%%%%%%%%%%
\section{Star decompositions from thin independent sets} \label{sec:star_decomp}

We fix a degree $d \geq 3$ and an integer $k> d/2$. The main result of this section is Lemma~\ref{lem:sd} that provides a sufficient condition for the existence of $k$-star decompositions in (deterministic) $d$-regular graphs.

We will need the notion of \emph{in-regular orientation}. 
\begin{definition} \label{def:reg_orientation}
An \emph{in-regular orientation} of an undirected graph $H$ is an orientation with the property that each in-degree of the resulting directed graph is the same.

Note that if this in-degree is $\ell \in \Nb$, then $H$ needs to have average degree precisely $2 \ell$.
\end{definition}
Whether a graph has an in-regular orientation can be phrased as a flow problem, and 
%the integral flow theorem 
(the integral version of) the max flow min cut theorem 
%(Ford--Fulkerson)
provides the following characterization for the existence. It is an immediate consequence of e.g.~\cite[Theorem 1]{frank1976orien}, where orientations with general degree bounds were studied.
\begin{proposition} \label{prop:reg_orientation}
Suppose that $H$ has average degree $2 \ell$ for some $\ell \in \Nb$; i.e., $e(H)=\ell |V(H)|$. Then $H$ has an in-regular orientation if and only if every induced subgraph of $H$ has average degree at most $2\ell$, that is, for every $U \subseteq V(H)$: 
\begin{equation} \label{eq:bo_cond1}
e[U] \leq \ell |U| .
\end{equation}
Note that an equivalent way of writing \eqref{eq:bo_cond1} in graphs with $e(H)=\ell |V(H)|$ is:
\begin{equation} \label{eq:bo_cond2}
e\big[ V(H) \sm U \big] + 
e\big[U,V(H) \sm U \big] 
\geq \ell \big| V(H) \sm U \big| .
\end{equation}
\end{proposition}
Next we state a simple observation that was already pointed out in the introduction.
\begin{observation} \label{obs:sd}
Let $G$ be a $d$-regular graph. For a given integer $k>d/2$, the following objects are clearly equivalent:
\begin{itemize}
\item a $k$-star decomposition of the edge set $E(G)$;
\item an orientation of $G$ in such a way that each out-degree is $0$ or $k$;
\item an independent set $A$ of density $\ds \al_{d,k} = 1 - \frac{d}{2k}$ along with a $(d-k)$-in-regular orientation of the remaining graph $G[A^\comp]$ (so that each in-degree is $d-k$).
\end{itemize}
Note that the condition that the density of $A$ is $\al_{d,k}$ implies that $\frac{d|V(G)|}{2k}$ is an integer, which is indeed necessary if we want to decompose $|E(G)|=\frac{d|V(G)|}{2}$ edges into $k$-stars.
\end{observation}
Clearly, such a $(d-k)$-in-regular orientation can exist only if $A$
has the property that every degree of $G[A^\comp]$ is at least $d-k$. In other words, every vertex outside $A$ can have at most $k$ neighbors in $A$. This motivates the following definition.
\begin{definition} \label{def:thin}
Given a positive integer $\dhat < d$, we say that an independent set $A \subset V(G)$ is \emph{$\dhat$-thin} if every vertex of $G$ has at most $\dhat$ neighbors in $A$:
\[ \forall v \notin A: \; e[v,A] \leq \dhat .\]
\end{definition}
Now we are ready to state our sufficient condition for the existence of a $k$-star decomposition. %Among other conditions, we require $A$ to be $\dhat$-thin for some $\dhat<k$ rather than just to be $k$-thin (which is a necessary condition). 
\begin{lemma} \label{lem:sd}
Let $G$ be a $d$-regular graph on $N$ vertices and let $k>d/2$. Let $\al_{d,k}$ be as in \eqref{eq:al_dk}. Suppose that $G$ has an independent set $A$ satisfying the following conditions for a positive integer $\dhat < k$ and a real number $0<c<1$.
\begin{enumerate}[(i)]
\item $A$ is $\dhat$-thin with density $\al_{d,k}$.
\item For any set $U \subseteq A^\comp = V(G) \sm A$ with $|U|\leq cN$ we have
\[ e[U] \leq (d-k) |U| .\]
\item For any set $W \subseteq A^\comp$ with $|W| < (1-\al_{d,k}-c)N$ we have 
\[ e[W] \leq (k-\dhat) |W| .\]
\end{enumerate}
Then $G$ has a $k$-star decomposition.
\end{lemma}
\begin{proof}
We use Proposition~\ref{prop:reg_orientation} with $\ell=d-k$ to show that $H \defeq G[A^\comp]$ has an $\ell$-in-regular orientation. This would prove the statement, according to Observation~\ref{obs:sd}. 
% First of all, the average degree of $H$ is indeed $2 \ell$ because there are $d \al_{d,k} N$ edges incident to $A$.
Let $U \sqcup W = A^\comp$ be any partition of $V(H)=A^\comp$ into two sets. We need to check condition \eqref{eq:bo_cond1}. This is immediate from (ii) for every set $U$ with $|U|\leq cN$. So we may assume that $|U| > cN$. Then  
\[ |W| = |A^\comp| - |U| < \big( 1-\al_{d,k} - c \big)N ,\]
and hence we can use (iii). Since $A$ is $\dhat$-thin, we have for any $v \in A^\comp$ that 
\[ \deg_H(v) = d - e[v,A] \geq d - \dhat .\]
Using this and (iii), we get 
\[ e[W] + e[W,U] = \left( \sum_{v \in W} \deg_H(v) \right) - e[W] 
\geq (d-\dhat) |W| - (k-\dhat) |W| = (d-k) |W| ,\]
meaning that \eqref{eq:bo_cond2} holds for $\ell=d-k$. Recall that \eqref{eq:bo_cond2} is just an equivalent formulation of \eqref{eq:bo_cond1}, so the condition of Proposition~\ref{prop:reg_orientation} holds for every $U$, and the proof is complete.
\end{proof}

The reason why this lemma will be useful in the context of random regular graphs is that (for large $d$) $\Gc_{N,d}$ a.a.s.~has an $\eps d$-thin independent set with density only slightly smaller than $\alstar_d$; see Lemma~\ref{lem:thin}(b). As we will see in Section~\ref{sec:main_proof}, this result implies that condition (i) of Lemma~\ref{lem:sd} is a.a.s.~satisfied for $k \leq \kind_d-1$. Conditions (ii) and (iii) will follow from a basic estimate on the average degree of induced subgraphs of $\Gc_{N,d}$. 

\begin{remark} \label{rm:not_divisible}
The condition that $k$ divides $e(G)=Nd/2$ is hidden in (i) because otherwise $\al_{d,k}N$ is not an integer and hence no subset of density exactly $\al_{d,k}$ may exist. If this divisibility condition is not satisfied, then we can still conclude that there exists a collection of edge-disjoint $k$-stars covering all but at most $k-1$ edges of $G$, which is the best we can hope for in this case. We simply need to replace condition (i) with the following variant:
\begin{enumerate}[(i')]
\item $A$ is $\dhat$-thin with size $\lceil \al_{d,k}N \rceil$.
\end{enumerate}
Indeed, in this case $H=G[A^\comp]$ has $\ell |V(H)| - r$ edges for some nonnegative integer $r \leq k-1$. Condition \eqref{eq:bo_cond2} still implies \eqref{eq:bo_cond1}, while \eqref{eq:bo_cond1} implies that $H$ has an orientation such that each in-degree is at most $\ell$ \cite[Theorem 1]{frank1976orien}. This orientation can be thought of as $|A^\comp|$ edge-disjoint stars, each with at least $k$ edges. Since $e(G)=k|A^\comp|+r$, this yields a collection of edge-disjoint $k$-stars in $G$ with only $r \leq k-1$ edges not covered.
\end{remark}

%%%%%%%%%%%%%%%%%%%%%%%%%%%%%%%%%%%%%%%%%%%%%
\section{Independent sets and counting} 
\label{sec:ind_set_counting}

In this section we review the known facts and techniques we will need in our proofs in subsequent sections. We start with a brief overview of the various results (exact formulas, upper bounds) regarding the independence ratio of random regular graphs.

\subsection{Bounds on the independence ratio}
\label{sec:ind_ratio}

As mentioned in the introduction, for each degree $d \geq 3$, there exists a constant $\alstar_d$ such that the independence ratio of the random $d$-regular graph $\Gc_{N,d}$ converges in probability to $\alstar_d$ as $N \to \infty$. This was proved in \cite{bayati2013interpolation}. %In fact, a standard Azuma-type argument shows that the independence ratio is highly concentrated around its expectation for any fixed $d$ and $N$.

The first upper bound on $\alstar_d$ was given by Bollob\'as \cite{bollobas1981independence}. We denote this bound by $\alfm_d$ and will refer to it as the \emph{first moment bound}. Let 
\begin{equation} \label{eq:phi}
\varphi_d(\al) \defeq h(\al) + \frac{d}{2} \, h(1-2\al) - (d-1) h(1-\al) 
\end{equation}
and let $\alfm_d$ be the unique root of $\varphi_d$ on $(0,1/2)$. Then $\alstar_d \leq \alfm_d$ for every $d \geq 3$. Asymptotically, as $d \to \infty$, we have $\alfm_d = \big(1 + o_d(1) \big) \frac{2\log d}{d}$ as $d \to \infty$. A more precise approximation can be found in \cite[formula (4)]{ding2016maximum}:
\begin{equation} \label{eq:alfm_approx}
\alfm_d = \frac{2}{d} \left( \log d - \log \log d + 1 - \log 2 
+ \Oc\left( \frac{\log \log d}{\log d} \right) \right) .
\end{equation}
A decade after Bollob\'as's bound, it was shown by {\L}uczak and Frieze \cite{frieze1992independence} that 
\begin{equation} \label{eq:alstar_lower_bound}
\alstar_d \geq \frac{2}{d} \bigg( \log d - \log \log d + 1 - \log 2 -o_d(1) \bigg) .
\end{equation}

So the first moment bound $\alfm_d$ is asymptotically optimal. However, it is not sharp for any given $d$. The reason behind this is that independent sets form clusters. Such a cluster may contain a large number of independent sets that differ little from each other. This leads to the phenomenon that, although most $d$-regular graphs do not contain independent sets of density $\al=\alstar_d+\eps$, a relatively small number of $d$-regular graphs contain exponentially many, causing the expected number to be above $1$, and hence the first moment bound fails to be sharp. These clusters can be described by the so-called \emph{frozen configurations}. In their breakthrough paper \cite{ding2016maximum} Ding, Sun, and Sly determined the expected number of such frozen configurations for large $d$ and obtained the improved upper bound $\alfc_d$. In fact, they proved that 
\begin{equation} \label{eq:dss}
 \alstar_d=\alfc_d 
\quad  \text{for sufficiently large } d .
\end{equation}
See \cite[Theorem 1 and formulas (1) and (2)]{ding2016maximum} for the exact result and the actual definition of $\alfc_d$. The improvement compared to the first moment bound is as follows (see \cite[Theorem 3.1]{ding2016maximum}):
\begin{equation} \label{eq:alfc_vs_alfm}
\alfc_d = \alfm_d - \left( \frac{2}{e} \, \frac{\log d}{d} \right)^2
\left( 1 + \Oc \left( \frac{\log \log d}{\log d} \right) \right) .
\end{equation}

Although not needed in this paper, we say a few words about another approach---rooted in statistical physics---to determining $\alstar_d$. There is a non-rigorous technique called the \emph{cavity method} that was already used in \cite{barbier2013hardcore} to predict the Ding--Sly--Sun result. More precisely, a 1-RSB (1-step replica symmetric breaking) formula $\alrsb_d$ was obtained through the cavity method, and it was conjectured that $\alstar_d=\alrsb_d$ for every $d \geq 20$. (This is still widely expected to hold true.) In fact, there is a rigorous technique called the \emph{interpolation method} that proves one direction (see \cite{lelarge2018replica,harangi2023rsb} for details):
\[ \alstar_d \leq \alrsb_d  
\quad \text{for every } d \geq 3 .\]
There is a phase transition below $d=20$, and a so-called full-RSB picture is expected in that range; see \cite{harangi2023rsb} for improved RSB bounds for $d<20$.

As for concrete (implicit) formulas, see \cite[Formula (1)]{harangi2023rsb} for the ``replica symmetric bound'' $\alrs_d$ and \cite[Formula (2)]{harangi2023rsb} for the 1-RSB bound $\alrsb_d$. These formulas---although they come from a very different approach---can be shown to be the same as $\alfm_d$ and $\alfc_d$: 
\[ \alfm_d = \alrs_d \; (\forall d \geq 3) 
\quad \text{and} \quad
\alfc_d = \alrsb_d \; (\forall d \geq 20) .\]

\subsection{First moment method in random regular graphs}
\label{sec:first_moment}

Recall that $\Gc_{N,d}$ stands for a random $d$-regular graph on $N$ vertices, that is, a uniform random graph among all $d$-regular simple graphs on the vertex set $\{1,\ldots,N\}$. 

Next we review the connection of $\Gc_{N,d}$ to the so-called \emph{configuration model}. Given $N$ vertices, each with $d$ ``half-edges'', the configuration model picks a random matching/pairing of these $Nd$ half-edges, resulting in $Nd/2$ edges. We denote the corresponding random graph as $\Gb_{N,d}$. Note that $\Gb_{N,d}$ is $d$-regular but it may have loops and multiple edges. A well-known fact is that if $\Gb_{N,d}$ is conditioned to be simple, then we get back $\Gc_{N,d}$. Moreover, for any $d$, the probability that $\Gb_{N,d}$ is simple converges to a positive $p_d$ as $N \to \infty$. It follows that if $\Gb_{N,d}$ a.a.s.~has a certain property, then so does $\Gc_{N,d}$.

Suppose that we want to show that $\Gc_{N,d}$ a.a.s.~does not contain a certain object, that is, $\Gc_{N,d}$ contains this object with probability $o_N(1)$. This probability is often exponentially small in $N$, and in many cases this can be proved by a standard first moment argument.

Let $Z_N$ denote the number of certain objects in $\Gb_{N,d}$. For instance, the \emph{object} may be an independent set of a given (approximate) density $\al$. The expectation $\Eb Z_N$ often grows or decays exponentially in $N$. By the exponential rate of growth we mean 
\[ \lim_{N \to \infty} \frac{\log(\Eb Z_N)}{N} .\]
If this limit exists, it may be regarded as the \emph{entropy of the object}. When negative, we can conclude that $\Eb Z_N$ is exponentially small in $N$, and hence so is $\Pb(Z_N>0) \leq \Eb Z_N$. Consequently, $\Gb_{N,d}$, and hence $\Gc_{N,d}$,  a.a.s.~does not contain such an object.

Actually, for a general class of objects, the limit can be expressed using Shannon entropy. Recall that the Shannon entropy of a discrete distribution $\mu$ (over some finite set $S$) is 
\[ H(\mu) \defeq \sum_{s \in S} h\big( \mu( \{s\} ) \big) ,\]
where\footnote{Throughout the paper $\log$ means natural logarithm.} 
\[ h(x) = 
\begin{cases} 
-x \log x  & \text{if } x\in (0,1]; \\
0 & \text{if } x=0 .
\end{cases} \]

For our purposes, the following simple setting will be sufficient. Suppose that we have finitely many vertex labels, and our object is a vertex-labeling with constraints describing which pairs of labels (and with what frequency) we may see on an edge. In essence, if $Z_N$ is the number of vertex-labelings of $\Gb_{N,d}$ with a prescribed (approximate) vertex and edge distribution ($\muvertex$ and $\muedge$, respectively), then 
\begin{equation} \label{eq:lim_log_Z}
\lim_{N \to \infty} \frac{\log(\Eb Z_N)}{N} = 
\frac{d}{2}H(\muedge) - (d-1) H(\muvertex) .
\end{equation}
More precisely, suppose that $G$ is a (deterministic) $d$-regular graph with a vertex labeling $\ell: V(G) \to L$ for some finite set $L$ of labels. The label $\ell(v)$ of a uniform random vertex $v$ of $G$ has a discrete distribution $\mu^{G,\ell}_\textrm{vertex}$ on $L$. Similarly, if we take a uniform random directed edge (i.e., an ordered pair $(u,v)$ of neighboring vertices), then the (joint) distribution of $\big( \ell(u),\ell(v) \big)$ is a discrete distribution $\mu^{G,\ell}_\textrm{edge}$ on $L \times L$. %exchangeable
Note that both marginals of $\mu^{G,\ell}_\textrm{edge}$ are equal to $\mu^{G,\ell}_\textrm{vertex}$. We want to consider labelings for which $\mu^{G,\ell}_\textrm{edge}$ is close to some prescribed distribution $\mu_\textrm{edge}$ on $L \times L$. By ``close'' we mean that, say, their total variation distance is less than $\delta_N$ for some positive sequence $\delta_N \gg 1/N $ converging to $0$, such as $\delta_N=N^{-1/2}$ or $\delta_N = 1/\log N$. (Then the same is automatically true for their marginals $\mu^{G,\ell}_\textrm{vertex}$ and $\mu_\textrm{vertex}$.)

Suppose that $\mu_\textrm{edge}$ is a disrtibution on $L \times L$ described by the probabilities $p_{ij} \; (i,j \in L)$ such that $p_{ij}=p_{ji}$. Then the marginal distribution $\mu_\textrm{vertex}$ has probabilities $p_i \defeq \sum_{j \in L} p_{ij}$ for $i \in L$. We define the random variable $Z_N$ as the number of labelings $\ell$ on $\Gb=\Gb_{N,d}$ such that $\dist_\textrm{TV}\big( \mu^{\Gb,\ell}_\textrm{edge} , \mu_\textrm{edge} \big)<\delta_N$. Then standard counting arguments show that 
\begin{equation} \label{eq:EZ_N}
\Eb Z_N = N^{\Oc(1)} \, \exp\left( N \left( \frac{d}{2}H(\muedge) - (d-1) H(\muvertex) \right) \right) ,
\end{equation}
and \eqref{eq:lim_log_Z} follows. For a rigorous proof see \cite[Lemma 4.1 and the proof of Theorem 4]{backhausz2018largegirth}.

For a specific example, let us consider independent sets $A$ of density $\al$. Such an $A$ can be described by a $\{0,1\}$-labeling with vertex distribution 
\begin{align*}
p_0 &= \; \al; \\
p_1 &= \; 1-\al,
\end{align*} 
and edge distribution 
\begin{align*}
p_{00} &= \; 0; \\
p_{01} = p_{10} &= \; \al; \\
p_{11} &= \; 1-2\al .
\end{align*} 
Then the entropy of this object, as in \eqref{eq:lim_log_Z}, is the following:
\begin{align*}
\varphi_d(\al) &= \frac{d}{2} \bigg( 2 h(\al) + h(1-2\al) \bigg) - (d-1)\bigg( h(\al) + h(1-\al) \bigg) \\
&= h(\al) + \frac{d}{2} \, h(1-2\al) - (d-1) h(1-\al) .
\end{align*}
It follows that if $\varphi(\al)<0$, then the independence ratio of $\Gc_{N,d}$ is  a.a.s.~less than $\al$. Note that this is the same function that we introduced in Section~\ref{sec:ind_ratio}. Recall that we defined $\alfm_d$ as the unique root of $\varphi_d$ on $(0,1/2)$ so that $\varphi_d(\al)<0$ for any $\al>\alfm_d$. This proves the first moment bound $\alstar_d \leq \alfm_d$, due to Bollobás \cite{bollobas1981independence}, mentioned in Section~\ref{sec:ind_ratio}.

%%%%%%%%%%%%%%%%%%%%%%%%%%%%%%%%%%%%%%%%%%%%%
\section{Thin independent sets} 
\label{sec:thin}

In this section we prove the a.a.s.~existence of $\dhat$-thin independent sets in random $d$-regular graphs, where $\dhat=\tau d$ for some positive constant $\tau$. To this end, given an independent set $A$, we need to bound the number of vertices with many neighbors in $A$. Our exact result is the following.

\begin{lemma} \label{lem:thin}
Fix a constant $0<\tau \leq 1$ and set $\dhat=\lceil \tau d \rceil$. Furthermore, let 
\begin{equation} \label{eq:beta_choice}
\be_d \defeq \left( \frac{\log d}{d}\right)^3 .
\end{equation}
Then for sufficiently large $d \geq d_0(\tau)$ the following hold a.a.s.~for $\Gc_{N,d}$.
\begin{enumerate}[(a)]
\item Suppose that $A$ is an independent set of $\Gc_{N,d}$ with density $\al$ in the range 
\begin{equation} \label{eq:al_range}
\alfm_d - \left( \frac{\log d}{d} \right)^2 
< \al < \alfm_d .
\end{equation}
Then there are at most $\be_d N$ vertices $v \notin A$ such that 
$e[v,A] \geq \dhat$.
\item There exists a $\dhat$-thin independent set in $\Gc_{N,d}$ with density at least 
\[ %\alstar_d - (1-\tau)d \be_d \geq 
\alstar_d - \frac{(\log d)^3}{d^2}. \]
Note that the bounds in the statements do not depend on $\tau$ but the threshold $d_0(\tau)$ does.
\end{enumerate}
\end{lemma}
The proof considers the expected number of the following objects in $\Gc_{N.d}$: an independent set $A$ of density $\al$ along with a set $B \subset A^\comp$ of density $\be$ such that $e[B,A]=\tau d|B|$, that is, a vertex in $B$ has $\tau d$ neighbors in $A$ on average. As we will see in Lemma~\ref{lem:phihat}, the exponential rate of this expectation is given by the following function.
\begin{definition} \label{def:phihat}
For 
$\al \in [0,1/2]$, 
$\be \in [0,1-2\al]$ and 
$\tau \in [0,1]$ let 
\begin{multline*}
\phihat_d(\al,\be,\tau) \defeq 
\frac{d}{2} \bigg( 2h(\be)+ 2\be \big( h(\tau)+h(1-\tau) \big)  
+ 2h(\al-\tau \be)
+ 2h(1-2\al-(1-\tau)\be) - h(1-2\al) \bigg) \\
- (d-1) \bigg(  h(\al)+h(\be)+h(1-\al-\be) \bigg) .
\end{multline*}
Note that for $\be=0$ we get back $\varphi_d$:
\[ \phihat_d(\al,0,\tau) = \varphi_d(\al) 
\quad \text{for every } \tau .\]
\end{definition}
\begin{lemma} \label{lem:phihat}
Suppose that $\phihat_d(\al,\be,\tau)<0$. Then $\Gc_{N,d}$ a.a.s.~does not contain a pair $A,B$, where $A$ is an independent set of density $\al+o(1)$, $B \subset A^c$ has density $\be+o(1)$, and $e[B,A] = \tau d |B| + o(N)$. 
\end{lemma}
\begin{proof}
We use the method outlined in Section~\ref{sec:first_moment}. Labeling the vertices in $A$, $B$, $(A \cup B)^\comp$ with $0,1,2$, respectively, we get the following vertex distribution $\mu_\mathrm{vertex}$:
\begin{align*}
p_{0} &= \; \al; \\
p_{1} &= \; \be; \\
p_{2} &= \; 1-\al-\be .
\end{align*} 
Its Shannon entropy is 
\[ H(\mu_\mathrm{vertex}) = h(\al)+h(\be)+h(1-\al-\be) . \]
The edge distribution is not completely determined by our constraints. The probabilities we know for sure are the following:
\begin{align*}
p_{00} &= \; 0; \\
p_{01} = p_{10} &= \; \tau \be ; \\
p_{02} = p_{20} &= \; \al-\tau \be .
\end{align*} 
The rest of the distribution has one degree of freedom: $p_{11}$, $p_{22}$, $p_{12}=p_{21}$ can be chosen in such a way that 
\[ p_{11}+p_{12}=p_1-p_{10}=(1-\tau)\be 
\quad \text{and} \quad 
p_{11}+p_{12}+p_{21}+p_{22}=1-2\al .\]
A well-known fact is that the ``independent coupling'' gives the largest entropy in such situations. More precisely, if 
\[ p_{11}+p_{12}=a_1 \; \text{ and } \;
p_{21}+p_{22}=a_2 \; \text{ with } \;
p_{12}=p_{21} , \]
then 
\[ h(p_{11})+h(p_{12})+h(p_{21})+h(p_{22}) \leq 
2h(a_1)+2h(a_2)-h(a_1+a_2) ,\]
with equality for $p_{ij}=a_i a_j/A$, where $i,j \in \{1,2\}$ and $A=a_1+a_2$. (Indeed, since $\log$ is a concave function, we have 
\[ h(A) + \sum_{i,j}h(p_{ij})-2\sum_i h(a_i) 
=  \sum_{i,j} p_{ij} \log\left( \frac{a_i a_j}{p_{ij}A} \right) 
\geq A \log\left( \sum_{i,j} \frac{a_i a_j}{A^2} \right) = 0 ,\]
where we used that $\sum p_{ij}=A$ and $\sum a_i a_j=A^2$.)

In our setting we have $a_1=(1-\tau)\be$ and $A=a_1+a_2=1-2\al$. Therefore, 
\begin{align*} 
H(\mu_\mathrm{edge}) &= \sum_{0 \leq i,j \leq 2} h(p_{ij}) \\
&\leq 2h(\tau \be) + 2h(\al-\tau \be) + 2h((1-\tau)\be) 
+ 2h(1-2\al-(1-\tau)\be) - h(1-2\al) .
\end{align*} 
Since 
\[ h(\tau \be)+h((1-\tau) \be)
= h(\be)+ \be \big( h(\tau)+h(1-\tau) \big) ,\]
we can write 
\[ H(\mu_\mathrm{edge}) \leq 
2h(\be)+ 2\be \big( h(\tau)+h(1-\tau) \big)  
+ 2h(\al-\tau \be)
+ 2h(1-2\al-(1-\tau)\be) - h(1-2\al) .\]
We conclude that 
\[ \frac{d}{2} H(\mu_\mathrm{edge}) 
- (d-1) H(\mu_\mathrm{vertex}) \leq 
\phihat_d(\al,\be,\tau) .\]
Recall that $\mu_\mathrm{edge}$ has one degree of freedom. Suppose that $\eta \defeq \phihat_d(\al,\be,\tau)<0$. According to Section~\ref{sec:first_moment} (and \eqref{eq:EZ_N} in particular) it follows that the expected number of labelings is at most $N^{\Oc(1)} \exp(N \eta)$ for any possible $\mu_\mathrm{edge}$. Since the number of terms is $\Oc(N)$, their sum is still exponentially small, and the proof is complete.
\end{proof}

\begin{proof}[Proof of Lemma~\ref{lem:thin}]
The main step is showing that $\phihat_d(\al,\be_d,\tau)<0$. Heuristically, one expects $\phihat_d(\al,\be_d,\tau)$ to be somewhat smaller than $\varphi_d(\al)$, which, in turn, is only slightly above $\varphi_d(\alfm_d)=0$ provided that $\al$ is close to the first moment bound $\alfm_d$. The precise estimates are as follows.

First note that $h'(x)=-\log(x)-1 \geq -1 $ is a monotone decreasing function on $(0,1)$. It follows that 
\[ (b-a) h'(a) \geq h(b)-h(a) \geq (b-a) h'(b) \geq -(b-a)   
\quad \text{for any } 0<a \leq b \leq 1 .\]
This observation yields the following inequalities:
\begin{align*}
h(1-\al-\be) &\geq h(1-\al)+\be \, h'(1-\al-\be) ;\\
h(1-2\al-(1-\tau)\be) &\leq h(1-2\al) + (1-\tau)\be ;\\
h(\al-\tau \be) &\leq h(\al) - \tau \be h'(\al) 
= h(\al) + \tau \be (\log(\al)+1) .
%\approx h(\al) - \tau \be \big( d - \log(d) -\log(2) - 1 \big) .
\end{align*}
Then, if $0<\tau \leq 1$ is a constant (not depending on $d$), we get  
\begin{equation} \label{eq:phihat_approx}
\phihat_d(\al,\be,\tau) \leq \varphi_d(\al) + 
h(\be) + \tau\be d \log(\al) + \Oc(\be d) .
%= \varphi_d(\al) 
%+ \be \bigg( \tau d \log(\al) - \log(\be) + \Oc(d) \bigg) . 
\end{equation}
Next we bound $\varphi_d$ near $\alfm_d$. Recall that $\varphi_d(\alfm_d)=0$ by definition. Differentiating \eqref{eq:phi} gives  
\begin{equation*} %\label{eq:phi_der}
- \varphi'_d(\al) = 
\log(\al) - d \log(1-2\al) + (d-1)\log(1-\al) 
\leq - d \log(1-2\al) \leq 2d \al .
\end{equation*}
Since $\al < \alfm_d < (2 \log d)/d$, it follows that $-\varphi'_d(\al) < 4 \log d$. (In fact, $-\varphi'_d(\al) \sim \log d$.) We conclude that for any $\al$ in the range \eqref{eq:al_range} we have 
\begin{equation} \label{eq:phi_bnd}
\varphi_d(\al) 
= \varphi_d(\al) - \underbrace{\varphi_d(\alfm_d)}_{=0} 
= -\int_{\al}^{\alfm_d} \varphi'_d  
< (\alfm_d - \al) 4\log d 
< \frac{4(\log d)^3}{d^2} .
\end{equation}
Furthermore, $\al < (2 \log d)/d$ gives 
\begin{equation} \label{eq:log_al} 
\log(\al)<  \log(2) +\log \log d -\log d < - \frac{\log d}{2} .
\end{equation}
Setting $\be=\be_d$ as in \eqref{eq:beta_choice} and using \eqref{eq:phihat_approx}, \eqref{eq:phi_bnd}, \eqref{eq:log_al}
%$h(\be_d) \leq \frac{(\log d)^3}{d^3} 3 \log(d)$ and 
we get 
\[ \phihat_d(\al,\be_d,\tau) \leq 
\frac{4(\log d)^3}{d^2} 
+ 3\frac{(\log d)^4}{d^3}  
- \frac{\tau}{2} \, \frac{(\log d)^4}{d^2} 
+ \Oc\left( \frac{(\log d)^3}{d^2} \right) < 0 \]
for sufficiently large $d$. 

Then, by Lemma~\ref{lem:phihat}, we conclude that for any independent set $A$ with density $\al$ in the range \eqref{eq:al_range}, there must be at most $\be_d N$ vertices with at least $\dhat=\lceil \tau d \rceil$ neighbors in $A$, and the proof of part (a) is complete.

For (b), we first note that, due to \eqref{eq:alfc_vs_alfm}, $\al=\alfc_d$ falls into the range \eqref{eq:al_range} for large enough $d$. Then the same is true for $\al \defeq \alfc_d-\eps$ if $\eps>0$ is small enough.

According to the Ding--Sly--Sun result \eqref{eq:dss}, 
we have $\alstar_d=\alfc_d$ for sufficiently large $d$. So $\Gc_{N,d}$ a.a.s.~has an independent set $A$ of density $\al<\alfc_d=\alstar_d$. Take such an $A$, and for any vertex $v \notin A$ with $e[v,A] \geq \dhat+1$, remove $e[v,A] - \dhat$ neighbors of $v$ from $A$. For each $v$, we remove at most $d-\dhat \leq d - \tau d = (1-\tau)d$ vertices. By part (a), the number of $v$'s is at most $\be_d N$. In total, we removed at most 
\[ (1-\tau)d\be_d N \leq d\be_d N = \frac{(\log d)^3}{d^2} N \]
vertices. We claim that the remaining independent set $A'$ is $\dhat$-thin. On the one hand, $e[u,A']=0$ for any removed vertex $u \in A \sm A'$. On the other hand, for $v \notin A$, either $e[v,A'] \leq e[v,A] \leq \dhat$ in the first place, or if $e[v,A] > \dhat$, then after removals we have $e[v,A'] \leq e[v,A] - \big( e[v,A] - \dhat \big) = \dhat$. Taking $\eps \to 0$, we get a $\dhat$-thin independent set $A'$ with the claimed density.
\end{proof}
Although not needed for our main result Theorem~\ref{thm:main}, we include here a lemma that we will use in Section~\ref{sec:specific}, where specific degrees are considered.
\begin{lemma} \label{lem:thin_plus}
Let $d$ and $\al$ be such that $\Gc_{N,d}$ a.a.s.~contains an independent set of density $\al$. Suppose that the following condition is satisfied for some integers $k>\dhat$: 
\[ \text{for } \tau_+ \defeq \frac{\dhat+1}{d} 
\text{ and } 
\bemax \defeq \inf \big\{ \be>0 \, : \, \phihat_d(\al,\be,\tau_+) < 0 \big\} \]
it holds that $(d-\dhat)\bemax < \al-\al_{d,k}$, 
or the following weaker condition holds: 
\begin{equation} \label{eq:best_condition}
 (\tau d - \dhat)\be < \al-\al_{d,k} 
\quad \text{for any $\be\leq \bemax$ and $\tau \geq \tau_+$ with } 
\phihat_d(\al,\be,\tau) \geq 0 .
\end{equation}
Then $\Gc_{N,d}$ a.a.s.~contains a $\dhat$-thin independent set of density $\al_{d,k}$. 
\end{lemma}
\begin{proof}
Let $A$ be an independent set of density $\al$ and set 
\[ B \defeq 
\big\{ v \notin A \, : \, e[v,A] \geq \dhat+1 \big\} .\]
Furthermore, set 
\[ \be \defeq \frac{|B|}{N} 
\quad \text{and} \quad 
\tau \defeq \frac{e[B,A]}{d|B|} .\]
We clearly have $\tau \geq \tau_+$ and we may assume that $\be \leq \bemax$ and $\phihat(\al,\be,\tau) \geq 0$, otherwise such a situation may occur in $\Gc_{N,d}$ only with exponentially small probability by Lemma~\ref{lem:phihat}. Therefore, $(\tau d - \dhat)\be < \al-\al_{d,k}$ holds by our assumption. So if we remove $e[v,A]-\dhat$ neighbors of every $v \in B$ from $A$, then the remaining set $A'$ will have density greater than $\al_{d,k}$. Moreover, $A'$ is $\dhat$-thin by the same argument as at the end of the previous proof.
\end{proof}

%%%%%%%%%%%%%%%%%%%%%%%%%%%%%%%%%%%%%%%%%%%
\section{Proof of the main theorem}
\label{sec:main_proof}

In this section we present the proof of our main result (Theorem~\ref{thm:main}). It is based on Lemma~\ref{lem:sd} and Lemma~\ref{lem:thin}. Additionally, we will need an upper bound on the average degree of an induced subgraph (with a given size) in random regular graphs. Later we will consider this problem in more detail; see Lemma~\ref{lem:av_deg} for sharper bounds. For the purposes of Theorem~\ref{thm:main}, the following (rather weak) bound will suffice.
\begin{proposition} \label{prop:weak}
There exists $\varrho<1$ such that, for sufficiently large $d$, it holds a.a.s.~for $\Gc=\Gc_{N,d}$ that 
\[ e[U] \leq \frac{\varrho d |U|}{2}
\quad \text{for every } U \subset V(\Gc) \text{ of size at most $N/2$,} \]
that is, the average degree of $\Gc[U]$ is at most $\varrho d$.
\end{proposition}
\begin{proof}
This is a simple consequence of a result of Bollob\'as: in \cite{bollobas1988isoperimetric} a lower bound $i^\star_d$ was given on the isoperimetric number (i.e., the Cheeger constant) of random $d$-regular graphs. Specifically, it holds a.a.s.~for $\Gc_{N,d}$ that for any $U$ with $|U| \leq N/2$ we have $e[U,U^\comp] \geq i^\star_d |U|$, where $i^\star_d = d/2-\Oc\big(\sqrt{d}\big)$. Then 
\[ e[U] = \frac{1}{2} \big( d|U| - e[U,U^\comp] \big) 
\leq \frac{d-i^\star_d}{2} |U| 
\leq \frac{d-d/2+\Oc\big(\sqrt{d}\big)}{2} |U| ,\]
and the claim follows for any $1/2<\varrho<1$. 
\end{proof}
\begin{proof}[Proof of Theorem~\ref{thm:main}]
Let $\dhat \defeq \lceil \tau d \rceil < \tau d+1$ for a small enough $\tau>0$ to be specified later. According to Lemma~\ref{lem:thin}, $\Gc_{N,d}$ a.a.s.~contains a $\dhat$-thin independent set of density at least
\[ \al \defeq \alstar_d - \frac{(\log d)^3}{d^2} .\]
We claim that a $k$-star decomposition a.a.s.~exists for any $k>d/2$ with $\al_{d,k}\leq \al$. Fix such a $k$. Then $\Gc_{N,d}$ contains a $\dhat$-thin independent set of density exactly $\al_{d,k}$ provided that $2k$ divides $Nd$. Therefore, condition (i) of Lemma~\ref{lem:sd} is satisfied. We claim that conditions (ii) and (iii) are satisfied as well.

Choose $\varrho<1$ as in Proposition~\ref{prop:weak}. Note that $\kind_d \sim d/2+\log d$ so we certainly have $d/2<k \leq \kind_d < d/2 + 2\log d$, hence  
\[ d-k > \frac{d}{2} - 2\log d 
\quad \text{and} \quad 
k-\dhat > \frac{d}{2} - \tau d - 1 .\]
Both are larger than $\varrho d/2$ if we choose $d$ sufficiently large and $\tau>0$ sufficiently small. This means that, setting $c=1/2$ in Lemma~\ref{lem:sd}, conditions (ii) and (iii) follow immediately from Proposition~\ref{prop:weak}. Therefore, the $k$-star decomposition indeed exists a.a.s. When $k$ does not divide $Nd/2$, we can use Remark~\ref{rm:not_divisible} to conclude that all but at most $k-1$ edges can be covered by edge-disjoint $k$-stars.

It remains to investigate the following: for what $k$ do we have 
\[ \al_{d,k} < \alstar_d - \frac{(\log d)^3}{d^2} .\]
Recall that 
\[ \kappa_d(\alstar_d) = \kappastar_d 
= \lfloor \kappastar_d \rfloor + \{ \kappastar_d \} 
= \kind_d + \{ \kappastar_d \} ,\]
where $\kappa_d$ is the inverse of $k \mapsto \al_{d,k}$; see \eqref{eq:kappa}. Then  
\[ \kappa_d(\al) - \kappa_d(\al-\eps) = 
\frac{d}{2} \, \frac{\eps}{(1-\al)(1-\al+\eps)} 
< \frac{d \eps}{2(1-\al)^2} .\]
Setting $\al=\alstar_d$ and $\eps=(\log d)^3/d^2$, for sufficiently large $d$ we get 
\[ \kappastar_d - \kappa_d(\alstar_d-\eps) 
< \frac{(\log d)^3}{2d \big( 1- (2\log d)/d \big)^2} 
= \frac{(\log d)^3d}{2 \big( d- 2\log d \big)^2} 
< \frac{(\log d)^3}{d} < 1 .
\]
In conclusion, $\kappa_d(\alstar_d-\eps) > \kind_d-1$ and it is even greater than $\kind_d$ in case $\{\kappastar_d\} > (\log d)^3/d$, which is precisely condition \eqref{eq:frac_part} of the theorem. Since 
\[ \al_{d,k} < \alstar_d-\eps 
\; \Longleftrightarrow \; 
k< \kappa_d(\alstar_d-\eps) ,\]
it follows that $k$-star decompositions exist for all $k \leq \kind_d-1$, and even for $k=\kind_d$ under condition \eqref{eq:frac_part}.

Finally, we show that the degrees $d$ for which condition \eqref{eq:frac_part} fails have asymptotic density $0$. Let $\gamma_d$ be such that 
\[ \alstar_d = \frac{2}{d} 
\big( \log d - \log \log d + 1 -\log 2 + \gamma_d \big) = 
\frac{2}{d} 
\left( \log\left( \frac{d}{2 \log d} \right) + 1 + \gamma_d \right).\]
We know from \eqref{eq:alfm_approx} and \eqref{eq:alstar_lower_bound} or \eqref{eq:alfc_vs_alfm} that $\gamma_d \to 0$ as $d \to \infty$. Then 
\[ \kappastar_d = \kappa_d(\alstar_d) 
= \frac{d}{2} \big( 1 + \alstar_d + \Oc((\alstar_d)^2) \big) 
= \frac{d}{2} + \log\left( \frac{d}{2 \log d} \right) + 1 + \gamma_d + \Oc\left( \frac{(\log d)^2}{d} \right) .\]
Suppose that \eqref{eq:frac_part} fails for some $d$, that is, $\{\kappastar_d\} \leq (\log d)^3/d$. Setting $m=2\lfloor \kappastar_d \rfloor - d - 2$ we get 
\[ \left| \frac{m}{2} - \log\left( \frac{d}{2 \log d} \right) \right| < \eps_m ,\]
where $\eps_m \to 0$ as $m \to \infty$. So every degree $d$ for which condition \eqref{eq:frac_part} fails must satisfy 
\[ \frac{d}{2 \log d} 
\in \left( e^{m/2-\eps_m}, e^{m/2+\eps_m} \right) 
\quad \text{for some } m \in \Zb .\]
It clearly follows that such degrees have asymptotic density $0$. The proof of Theorem~\ref{thm:main} is now complete.
\end{proof}

%%%%%%%%%%%%%%%%%%%%%%%%%%%%%%%%%%%%%
\section{Star decompositions for specific degrees}
\label{sec:specific}

What result does our approach yield for specific values of $d$? In the proof of Theorem~\ref{thm:main} we chose $\dhat \approx \tau d$ for some small positive constant $\tau$. What if $\dhat \approx d/2$ so that $\dhat$ is only slightly smaller than $k$ itself? For larger $\dhat$, condition (iii) of Lemma~\ref{lem:sd} becomes more difficult to satisfy, and we need to use more delicate estimates than the basic bound of Proposition~\ref{prop:weak}. The payoff is that condition (i) becomes easier to satisfy.

The more delicate estimates that we need to use will be stated precisely in Lemma~\ref{lem:av_deg} below. It provides a bound on the average degree of induced subgraphs that holds a.a.s.~for $\Gc_{N,d}$. For now, in order to keep the discussion as general as possible, we do not use the specific bound. Instead, we assume that we have a (continuous and strictly monotone increasing) function $g_d \colon (0,1) \to (2/d,1)$ with the following property:
\begin{multline} \label{eq:g_d_assumption}
\text{for every $x_0 \in (0,1)$ and $\eps>0$ it holds a.a.s.~for $\Gc=\Gc_{N,d}$ that}\\
\text{for any $U \subset V(\Gc)$ of density at most $x_0$,} \\
\text{the average degree of $\Gc[U]$ is less than $g_d(x_0)\cdot d+\eps$.}
\end{multline}
(Such a function $g_d$ will be given in Lemma~\ref{lem:av_deg} soon.)

Now we consider the following question: given a triple $d, k, \al$, \textbf{can our method guarantee a $k$-star decomposition} in random $d$-regular graphs \textbf{provided that an independent set of density $\al$ is guaranteed to exist?} We can decide this by the following procedure:
\begin{itemize}
\item set $\ds t_1 \defeq \frac{2(d-k)}{d}=1-\frac{2k-d}{d}$ and 
$x_1 \defeq g^{-1}_d(t_1)$;
\item set $x_2 \defeq 1-\al_{d,k}-x_1$ and 
$t_2 \defeq g_d(x_2)$;
\item set $\dhat \defeq \lfloor k - t_2 d/2 \rfloor$ so that $k-\dhat \geq t_2 d/2$;
\item check condition \eqref{eq:best_condition} of Lemma~\ref{lem:thin_plus}.
\end{itemize}
If the condition at the last step is verified, then our method guarantees the existence of a $k$-star decomposition. This follows immediately by combining Lemma~\ref{lem:thin}, Lemma~\ref{lem:sd} and our assumption \eqref{eq:g_d_assumption}.

To get an idea of the strength of our approach for specific degrees, we used a computer to run the above procedure for $\al=\alfc_d$ and $k=\lfloor \kappa_d(\al) \rfloor$. If $\alstar_d=\alfc_d$ indeed holds for $d \geq 20$ as conjectured, then the above $k$ is actually equal to $\kind_d$. Under this assumption the computer check confirmed the existence of a $\kind_d$-star decomposition for every $30 \leq d \leq 3000$ except for the following degrees:
\[ 31, 46, 48, 87, 89, 164, 166, 301, 303, 305, 550, 552, 554, 995, 997, 999, 1001, 1788, 1790, 1792 .\]
For these exceptional degrees $\kappastar_d=\kappa_d(\alstar_d)$ is too close to its integral part $\kind_d$ and our approach only guarantees a $k$-star decomposition for $k=\kind_d-1$.

\subsection*{The average degree of induced subgraphs}

Now we give the specific $g_d$ that we used. This result is essentially a reformulation of \cite[Theorem 3]{kolesnik2014isoperimetric}, where it is stated as a lower bound on the generalized \emph{isoperimetric number/Cheeger constant}. Although the result in \cite{kolesnik2014isoperimetric} concerns the case $x_0 \in (0,1/2]$, it holds true for every $x_0 \in (0,1)$ due to a symmetry around $1/2$.
\begin{lemma} \label{lem:av_deg}
For any $0 \leq x \leq t \leq 1$, let
\[ F_d(x,t) \defeq h(t x) + 2h( (1-t)x) + h(1-(2-t)x) 
- \left( 2- \frac{2}{d} \right) \big( h(x) + h(1-x) \big) .\]
For any fixed $x \in (0,1)$ the function $t \mapsto F_d(x,t)$ is continuous and strictly monotone decreasing on $[x,1]$ with a unique root that we denote by $g_d(t)$. Furthermore, $g_d$ continuously and strictly monotone increasingly maps $(0,1)$ onto $\left( \frac{2}{d},1 \right)$.

Suppose that 
\[ x_0 \in (0,1) 
\quad \text{and} \quad
t_0 \in \left( \frac{2}{d},1 \right) 
\quad \text{with} \quad 
g_d(x_0)<t_0 .\]
Then it holds a.a.s.~for $\Gc=\Gc_{N,d}$ that the average degree of the induced subgraph $\Gc[U]$ is less than $t_0 d$ for every $U \subset V(\Gc)$ with density $|U| \big/ N \leq x_0$.

Equivalently, we have the following lower bound on the ``generalized Cheeger constant'':
\begin{equation} \label{eq:gen_cheeger}
\min \left\{ \frac{e[U,U^\comp]}{|U|} \, : \, 
U \subset V(\Gc); \, 0< |U| \leq x_0 N \right\} 
> (1-t_0)d .
\end{equation}
\end{lemma}
For a complete proof we refer the reader to \cite{kolesnik2014isoperimetric}. Here we only include a brief outline. This is essentially a first moment bound, where one needs to count the subsets $U \subset \{1,\ldots, N\}$ for which $|U| \approx xN$ and the average degree of $G[U]$ is approximately $t d$, that is, $e[U] \approx t d |U|/2$. In the language of Section~\ref{sec:first_moment}, this corresponds to a $\{0,1\}$-labeling of the vertices with the following vertex and edge distributions $\mu_\textrm{vertex}$ and $\mu_\textrm{edge}$: 
\begin{align*}
p_0 &= \; x; \\
p_1 &= \; 1-x ;
\end{align*} 
and
\begin{align*}
p_{00} &= \; tx; \\ 
p_{01} = p_{10} &= \; (1-t)x; \\ 
p_{11} &= \; 1-(2-t)x .
\end{align*}
We get 
\begin{align*}
H(\mu_\textrm{vertex}) &= h(x)+h(1-x) ;\\
H(\mu_\textrm{edge}) &= h(t x) + 2h( (1-t)x) + h(1-(2-t)x) .
\end{align*}
Now we see where the function $F_d$ comes from: 
\[ \frac{d}{2} F_d(x,t) = \frac{d}{2} H(\mu_\textrm{edge}) 
- (d-1) H(\mu_\textrm{vertex}) .\]
According to \eqref{eq:lim_log_Z}, $\frac{d}{2} F_d(x,t)$ is equal to the exponential rate of growth of the expected number of subsets $U$ (with the given properties) in $\Gb_{N,d}$. In particular, if $g_d(x)<t$, then $F_d(x,t)<0$, so $\Gb_{N,d}$ a.a.s.~has no such subset, and hence the same is true for $\Gc_{N,d}$. 

To turn this argument into a rigorous proof, one basically needs two things. First, since this argument only works for subsets of linear size (when $x>0$), a little extra attention is needed to deal with subsets $U$ of sublinear size as well. Second, it is crucial that $g_d$ is increasing. (Otherwise we would need to replace $g_d(x_0)$ with $\max_{x \in (0,x_0]} g_d(x)$ in the bound.) One can prove this by showing that for fixed $t>2/d$ the function $x \mapsto F_d(x,t)$ has a unique root on $(0,t]$. 
%, which follows from the fact that it starts negative and ends positive, and it is convex in the first part of the interval and concave in the remaining part.

\bibliographystyle{plain}
\bibliography{refs}

\end{document}